%% file: main.tex
\documentclass[letterpaper, 10 pt, conference]{ieeeconf}  

\IEEEoverridecommandlockouts                              
\usepackage[T1]{fontenc}
\usepackage[utf8]{inputenc}
\usepackage{palatino} 
\usepackage{mathrsfs}
\usepackage{graphicx}
\usepackage{mathtools}
\usepackage{amsmath} 
\usepackage{amssymb}  
\usepackage{color}
\usepackage{enumerate}
\usepackage{multirow}
\usepackage{nomencl}
\usepackage{array,url}
\usepackage{epstopdf}
\graphicspath{{./Fig/}}

\newtheorem{theorem}{Theorem}[section]

\newtheorem{proposition}[theorem]{Proposition}

\newtheorem{remark}{Remark}[section]
\newtheorem{assumption}{Assumption}
\newtheorem{problem}{Problem}

\newcommand{\bmat}[1]{\begin{bmatrix}#1\end{bmatrix}}
\title{\LARGE \bf
The turnpike property in the maximum hands-off control
}

\author{Noboru Sakamoto${}^{1,\ast}$ and Masaaki Nagahara${}^{2,\S}$
\thanks{1: Faculty of Science and Engineering, Nanzan University, Yamazato-cho 18, Showa-ku, Nagoya, 464-8673, Japan}%
\thanks{2:Institute of Environmental Science and Technology, The University of Kitakyushu,
Hibikino 1-1, Wakamatsu-ku, Kitakyushu, Fukuoka, 808-0135, Japan}%
\thanks{$\ast$:  Supported, in part, by JSPS KAKENHI Grant Numbers JP26289128, JP19K04446 and by Nanzan University Pache Research Subsidy I-A-2 for 2020 academic year. {\tt\footnotesize E-mail:  noboru.sakamoto@nanzan-u.ac.jp}}%
\thanks{\S: This work has been partially funded by JSPS KAKENHI Grant Number 19H02301. {\tt\footnotesize E-mail:nagahara@ieee.org}\newline 
{\hrule width 30mm}  {\normalsize Submitted for IEEE Conference on Decision and Control 2020}  
}%
}

\begin{document}
\maketitle
\thispagestyle{empty}
\pagestyle{empty}

\begin{abstract}
This paper presents analyses for the maximum hands-off control using the geometric methods developed for the theory of turnpike in optimal control. First, a sufficient condition is proved for the existence of the maximum hands-off control for linear time-invariant systems with arbitrarily fixed initial and terminal points using the relation with $L^1$ optimal control. Next, a sufficient condition is derived for the maximum hands-off control to have the turnpike property, which may be useful for approximate design of the control. 
\end{abstract}
\input{introduction.tex}
\input{L1_existence}
\input{Exst_M-hoff}

\input{MHO_TP}
\input{Simulation}

\section{Conclusions}
In this paper, we considered the maximum hands-off control problem, which attracts much attention from the viewpoints of solving environmental problems \cite{Nagahara:16:ieeetac}, using the geometric analysis method developed for the theory of turnpike in optimal control. Using the equivalence of maximum hands-off control and $L^1$ optimal control under certain hypotheses, the existence of the maximum hands-off control for linear time-invariant systems is proved. Using the invariant manifold theory, it has been shown that the turnpike phenomenon appears in the maximum hands-off control under the conditions of normality and spectral conditions. The result may be useful from the fact that the occurrence of turnpike often leads to simplification of optimal control design.  
\bibliographystyle{IEEEtran}
\bibliography{biblio/Main.bib,biblio/OptimalControl.bib}
\end{document}

%% file: introduction.tex
\section{Introduction}  
Optimal control theory plays a significant role in modern control technologies and their applications to science and engineering. It provides an optimal strategy of inputs to alter dynamical systems so as for the inputs and system states to behave in an optimal way. The optimality often requires to minimize an integral of the inputs and states over the time of control process (Lagrange type). A typical form of the integral penalty (cost functional) is quadratic functions of inputs and states and design methods for this type is well-developed (see, e.g., \cite{Anderson:89:OCLQM}). 

From the viewpoint of better performance of controlled systems, non-quadratic cost functionals attract attention of theorists and practitioners in the control community. For instance, $L^1$ norm of control input is used to minimize the net amount of control effort and sometimes called the minimum fuel control problem (see, e.g., \cite{Athans:66:OC}). Recently, a control problem that maximizes the time interval over which control input is exactly zero has been proposed in \cite{Nagahara:16:ieeetac}. This problem is called maximum hands-off control and is potentially beneficial from the viewpoints of designing environmentally friendly systems \cite{Ikeda2016190,Ikeda2019_automatica,Nagahara2020_automatica}. For instance, this concept is useful and already used in electric/hybrid vehicles \cite{Chan2007704}, railway trains \cite{Liu2003917} and networked control systems \cite{Nagahara2014_ieeetac}. It is closely related to sparsity of signals, which is an active research area in system control and signal processing \cite{Donoho2006_ieeetit,Giselsson2013829}. 

In this paper, we explore further properties of the maximum hands-off control from the viewpoint of turnpike phenomenon. The turnpike phenomenon in optimal control was first observed in econometrics \cite{Dorfman:58:LPEA} and later, independently in control theory \cite{Wilde:72:ieeetac}. The turnpike theory says that the optimal control, when time-horizon is large enough, does not depend on the length of the horizon but depend only on the system and the cost functional except for thin boundary layers at the beginning and the end of the control horizon \cite{Carlson:91:IHOC,Zaslavski:06:TPCVOC}. One often encounters similar situation when traveling a long distance by a car; when the destination is far enough, "it will always pay to get on the turnpike to cover distance at the best rate of travel ..." \cite[Chapter~12]{Dorfman:58:LPEA}. The turnpike theory is recognized as useful tools to simplify the design process of optimal control \cite{Grune:13:automatica,Trelat:15:jde} and optimal shape design \cite{Porretta:16:INdAM,Lance2019_ifacnolcos}. The tool we employ in the present paper is based on invariant manifold theory in dynamical system theory such as (un)stable manifold and $\lambda$-lemma. In \cite{sakamoto:19:cdc}, they are applied to Hamiltonian systems derived from necessary condition of optimality in order to better understand the geometric nature of the turnpike and to give occurrence conditions for turnpike in terms of the locations of (un)stable manifolds. In the present paper, we consider optimal control problems where initial and terminal states are arbitrarily fixed and show that the turnpike phenomenon is observed in the maximum hands-off control under certain conditions, which can be used to simplify the construction of the control.  

The organization of the paper is as follows. In \S~\ref{sctn:L1}, a sufficient condition for the existence of $L^1$ optimal control is provided using the direct method of calculus of variations (see, e.g., \cite{Peypouquet:15:CONS}), which is a generalization of the result in \cite{Hajek1979jota}. \S~\ref{sctn:ex_MHO} shows that under the strong form of controllability condition and the conditions on initial and terminal states, the maximum hands-off control exists. In \S~\ref{sctn:HHO_TP}, it is shown that the turnpike phenomenon can be seen in the process of the maximum hands-off control. A simulation result is illustrated in \S~\ref{sctn:simulation}. 

%% file: L1_existence.tex
\section{Existence of $L^1$ optimal control}\label{sctn:L1} Let us consider an $n$-dimensional linear time-invariant system with $m$ inputs $u=(u^1,\ldots,u^m)^\top$
\begin{equation}
    \dot x = Ax +Bu,\ x(0)=x_0.\label{eqn:lti}
\end{equation}
Let $T$ be a given positive constant. For the optimal control problem defined below, the control set for (\ref{eqn:lti}) is taken as the Banach space $L^1((0,T);\mathbb{R}^m)$, or $L^1$ in short, the set of $\mathbb{R}^m$-valued measurable functions over $[0,T]\subset\mathbb{R}$ with $\int_0^T|u(t)|\,dt<\infty$. Also we introduce Banach spaces $L^2((0,T);\mathbb{R}^m)$ and $L^\infty((0,T);\mathbb{R}^m)$ or, $L^2$ and $L^\infty$, by the sets of $\mathbb{R}^m$-valued measurable functions with 
\[
\int_0^T|u(t)|^2\,dt \quad \text{and}\quad \max_{1\leqslant i\leqslant m}\underset{[0,T]}{\mathrm{ess\, sup}}|u^i(t)|
\]
being finite, respectively. The norms for $L^1$, $L^2$ and $L^\infty$ will be denoted by $\|\cdot\|_{L^1}$, $\|\cdot\|_{L^2}$ and $\|\cdot\|_{L^\infty}$, respectively. Finally, let us denote the closed unit ball in $L^\infty$ by $B_\infty$. 

The cost functional to be considered in this section is
\begin{equation}
J[u] = \int_0^T x(t)^\top Q x(t)+|u(t)|\,dt,\label{eqn:cost_L1}
\end{equation}
where $Q$ is a real nonnegative definite matrix. 
\begin{problem}\label{pbm:l1_optimal}
Given $T>0$ and $x_0$, $x_f\in\mathbb{R}^n$ for system (\ref{eqn:lti}), find a control $u$ that minimizes $J[u]$ over all control inputs in $L^1 \cap B_\infty$ that take the initial state $x_0$ to $x_f$ at $t=T$. 
\end{problem}
 
 Note that for finite $T$, we have 
 \begin{equation}
 L^\infty((0,T);\mathbb{R}^m)\subset L^2((0,T);\mathbb{R}^m)\subset L^1((0,T);\mathbb{R}^m)\label{eqn:L-inclusions}
 \end{equation}
 and therefore, for Problem~\ref{pbm:l1_optimal}, it suffices to look for controls in $B_\infty$. 
 
 Now, the main result of this section is stated as follows. 
 \begin{theorem}\label{thm:L1_existence}
 Suppose that there is a control $u\in B_\infty$ for (\ref{eqn:lti}) taking its initial state $x_0$ at $t=0$ to $x_f$ at $t=T$. Then, there exists an optimal control for Problem~\ref{pbm:l1_optimal}.
 \end{theorem}
\begin{proof}
Let $\varphi:\mathbb{R}^n\to\mathbb{R}\cup\{+\infty\}$ be the indicator function for $\{x_f\}$, namely, $\varphi(x_f)=0$ and $\varphi(x)=+\infty$ for $x\neq x_f$. Note that $\varphi$ is lower semi-continuous since $\{x_f\}$ is a closed set. Define a modified cost functional 
\[
\bar{J}[u] = \varphi(x(T))+\int_0^T x(t)^\top Q x(t)+|u(t)|\,dt.
\]
Considering $J[u]$ for controls that take the initial state to $x=0$ is equivalent to minimizing $\bar{J}[u]$ without the constraint of $x(T)=x_f$. We shall show that there exists a $\bar{u}\in B_\infty$ such that 
\begin{equation}
\bar{J}[\bar u] = \inf_{u\in B_\infty} \bar{J}[u].\label{eqn:end_of_pf}
\end{equation}
From the hypothesis there exists a sequence of controls $\{u_k\}\subset B_\infty$ such that $\bar{J}[u_k]<\infty$ and $\lim_{k\to\infty}\bar{J}[u_k]=\inf_{B_\infty} \bar{J}$. \\
(Step 1) We prove that up to subsequence, $\{u_k\}$ weakly converges to a $\bar u\in B_\infty$. From (\ref{eqn:L-inclusions}), $\{u_k\}$ is a bounded sequence in $L^2$ and therefore, up to subsequence, $\{u_k\}$ weakly converges to a $\bar u\in L^2$. From the weak convergence, for any Borel set $D\subset [0,T]\subset\mathbb{R}$, 
\[
\int_0^T\chi_D(t)u_k^i(t)\,dt\to \int_0^T\chi_D(t)\bar{u}^i(t)\,dt,\ i=1,\ldots,m
\]
as $k\to\infty$, where $\chi_D$ is the characteristic function for $D$. However, since $u_k\in B_\infty$, one obtains $\left| \int_Du_k^i(t)\,dt \right|\leqslant\mu(D)$ for all $k\in\mathbb{N}$, $i=1,\ldots,m$, where $\mu$ is the Lebesgue measure. Taking limit $k\to\infty$ yields $\left| \int_D\bar{u}^i(t)\,dt \right|\leqslant\mu(D)$, which shows that
\[
|\bar{u}^i(t)|\leqslant1 \quad \text{a.e.}\quad t\in[0,T], \ i=1,\ldots,m
\]
since $D\subset[0,T]$ is arbitrary. \\
(Step 2) Let $x_k$ be the solution of (\ref{eqn:lti}) corresponding to $u_k$. Then, it can be shown that $\{x_k\}$ is uniformly bounded and equicontinuous. Let $\bar x$ be the solution of (\ref{eqn:lti}) for $\bar u$. Up to subsequence, using Ascoli-Arzel\'a Theorem, $\{x_k\}$ uniformly converges to $\bar x$. The detail of this step is omitted. \\
(Step 3) We show that, up to subsequence, 
\begin{equation}
\|\bar u\|_{L^1}\leqslant \liminf_{k\to\infty}\|u_k\|_{L^1}\label{ineq:liminf}
\end{equation}
and (\ref{eqn:end_of_pf}) holds. Note that $\|u_k\|_{L^1}$ is bounded since $u_k\in B_\infty$. From Hahn-Banach Theorem, there is a bounded linear functional $l:L^1\to\mathbb{R}$ such that $\langle l, \bar{u} \rangle=\|\bar u\|_{L^1}$ and $\langle l,u\rangle\leqslant\|u\|_{L^1}$ for $u\in L^1$. It then holds that 
\begin{align*}
    \|\bar u\|_{L^1}&= \langle l,\bar{u}-u_k\rangle +\langle l,u_k\rangle\\
        &\leqslant \langle l, \bar{u}-u_k\rangle +\|u_k\|_{L^1},
\end{align*}
which yields (\ref{ineq:liminf}) from the weak convergence of $\{u_k\}$ to $\bar{u}$. So far, we have shown that 
\begin{gather*}
    \varphi(\bar{x}(T))\leqslant\liminf_{k\to\infty}\varphi(x_k(T))\\
    \int_0^Tx_k(t)^\top Qx_k(t)\,dt\to\int_0^T\bar{x}(t)^\top Q\bar{x}(t)\,dt,\ k\to \infty.
\end{gather*}
It hence follows that 
\begin{align*}
    \bar{J}[\bar u] &=\varphi(\bar{x}(T))+\int_0^T\bar{x}(t)^\top Q\bar{x}(t)+|u(t)|\,dt\\
    &\leqslant \liminf_{k\to\infty}\bar{J}[u_k]=\inf_{u\in B_\infty}\bar{J}[u],
\end{align*}
which completes the proof.
\end{proof}

%% file: Exst_M-hoff.tex
\section{A sufficient condition for the existence of maximum hands-off control}\label{sctn:ex_MHO}
Based on the result in the previous section on $L^1$ optimal control, this section considers the maximum hands-off control or optimal sparse control, which is defined as follows. 
\begin{problem}[Maximum hands-off control]\label{pbm:mhandsoff}
Let us consider system (\ref{eqn:lti}). For given $T>0$, $x_0$ and $x_f\in\mathbb{R}^n$, find a control $u$ that minimizes $\mu(\mathrm{supp}(u))$ over all control inputs in $B_\infty$ that take the initial state $x_0$ at $t=0$ to $x_f$ at $t=T$.  
\end{problem}

In \cite[Theorem 8]{Nagahara:16:ieeetac}, it is shown that $L^1$ optimization can be used for maximum hands-off solution under normality condition, a sufficient condition for which is explicitly obtained for (\ref{eqn:lti}) in \cite{Athans:66:OC}. Roughly speaking, system (\ref{eqn:lti}) is called normal if its $L^1$ optimal control takes values $\pm1$or 0 for almost all $t\in[0,T]$.  
\begin{assumption}[A sufficient condition for normality]\label{assm:normal}
For (\ref{eqn:lti}), all the pairs $(A,b_j)$, $j=1,\ldots,m$, are controllable and $A$ is nonsingular. 
\end{assumption}

%
Additionally, let us introduce notations to specify spectral condition of system (\ref{eqn:lti}). For a matrix $A\in\mathbb{R}^{n\times n}$, let $\mathscr{L}^{+}(A)$ ($\mathscr{L}^{-}(A)$) denote the generalized eigenspace for the eigenvalues of $A$ in the open left-half (right-half) plain in $\mathbb{C}$ and let $\mathscr{L}^{0+}(A)= \mathscr{L}^{+}(A)\oplus\mathscr{L}^{0}(A)$, $\mathscr{L}^{0-}(A)= \mathscr{L}^{-}(A)\oplus\mathscr{L}^{0}(A)$, where $\mathscr{L}^{0}(A)$ is the generalized eigenspace for the eigenvalues on the imaginary axis. 
\begin{theorem}\label{thm:existence_sparse_opt} Suppose that Assumption~\ref{assm:normal} holds. If $x_0\in \mathscr{L}^{0-}(A)$ and $x_f\in \mathscr{L}^{0+}(A)$, then for sufficiently large $T$, a solution for Problem~\ref{pbm:mhandsoff} exists.  
\end{theorem}
\begin{proof}
We show that an $L^1$ optimal control exists for Problem~\ref{pbm:l1_optimal} with $Q=0$. From the controllability of $(A,B)$, there exist $r>0$ and $t_0>0$ such that for all points in $|x|<r$ there exist controls $u(t)$ that take them to the origin within $[0,t_0]$ and satisfy $|u(t)|\leqslant1$ for $t\in[0,t_0]$. Also from $x_0\in \mathscr{L}^{0-}(A)$, there exists an input with $|u(t)|\leqslant1$ such that the corresponding state starting at $x_0$ at $t=0$ enters $|x|\leqslant r$ within a finite time, say, $t_1>0$. By considering $\dot x = -Ax$ and using the condition $x_f\in \mathscr{L}^{0+}(A)$, it is shown that for $x_0$, $x_f$, there exists a control $u(t)$ with $|u(t)|\leqslant1$ that takes $x_0$ to $x_f$ at $t=T$ if $T>2(t_0+t_1)$. Then, Theorem~\ref{thm:L1_existence} applies to show that there exists a solution to Problem~\ref{pbm:l1_optimal}. 
Now we use the results in \cite[Chapter~6]{Athans:66:OC} with the normality condition in Assumption~\ref{assm:normal} to conclude that this $L^1$ optimal control is bang-off-bang, namely, it takes only three values of $\pm1$ and $0$ almost everywhere. 
Therefore, from Theorem~8 in \cite{Nagahara:16:ieeetac}, this $L^1$ optimal control is actually a maximum hands-off control. 
\end{proof}

%% file: MHO_TP.tex
\section{Turnpike phenomenon in maximum hands-off control}\label{sctn:HHO_TP} 
As we have seen in the previous section, using $L^1$ optimal control theory, it is possible to provide a condition for the existence of maximum hands-off control. Theorem~\ref{thm:existence_sparse_opt}, however, provides little information on how to construct it. The present section shows that under certain conditions, the optimal control exhibits the turnpike phenomenon, from which one often deduces approximate designs. 

Here we provide only a basic definition and facts on turnpike property. For more detail, we refer to \cite{Carlson:91:IHOC,Zaslavski:06:TPCVOC}. 
The pair of the optimal control $u_T$ and corresponding states $x_T$ for (\ref{eqn:lti})-(\ref{eqn:cost_L1}) is said to have the turnpike property if for any $\varepsilon>0$, there exists an $\eta_\varepsilon>0$ such that 
\[
\mu( \left\{t\geqslant0\,|\,|u_T (t)|+|x_T(t,x_0)|>\varepsilon  \right\} ) <\eta_\varepsilon
\]
for all $T>0$, where $\eta_\varepsilon$ depends only on $\varepsilon$, $A$, $B$, $x_0$, and $Q$. In \cite{Porretta:13:sicon}, the turnpike inequality condition is proposed which requires for any $T>0$, $u_T$ and $x_T$ to satisfy 
\[
|u_T (t)|+|x_T(t,x_0)|\leqslant 
    K\left[e^{-a t}+e^{-a(T-t)}\right]
\]
for all $t\in[0,T]$ and some constants $K>0$, $a>0$ which are independent of $T$. The turnpike inequality condition is known to be sufficient for the turnpike property. 
\subsection{Review of geometric turnpike analysis via invariant manifold theory}
This subsection summarizes the geometric framework in \cite{sakamoto:19:cdc} which will be useful for $L^1$ optimal control analysis and subsequently for maximum hands-off control. Let us consider a nonlinear dynamical system of the form
\begin{equation}
    \dot z= f(z),\label{eqn:dyn_sys}
\end{equation}
where $f:\mathbb{R}^N\to\mathbb{R}^N$ is a class of functions satisfying the following assumptions.
\begin{assumption}\label{assm:hyp}
\begin{enumerate}[(i)]
    \item $f(0)=0$.
    \item $f$ is locally $C^1$ class around $z=0$.
    \item $f$ is hyperbolic at $z=0$, namely, $(\partial f/\partial z)(0)\in\mathbb{R}^{N\times N}$ has $k$ eigenvalues with strictly negative real parts and $N-k$ eigenvalues with strictly positive real parts.
    \item (\ref{eqn:dyn_sys}) admits unique Carath\'eodory solutions for all initial conditions (see, e.g., \cite[Section~I.5]{Hale:73:ODE}).
\end{enumerate}
\end{assumption}
 It is known, as {\em the stable manifold theorem}
, that there exist continuous manifolds $S$ and $U$, called {\em stable manifold} and {\em unstable manifold} of (\ref{eqn:dyn_sys}) at $0$, respectively, defined by 
\begin{align*}
S&:=\{z\in\mathbb{R}^N\,|\, \varphi(t,z)\to0 \text{ as } t\to \infty\},\\
U&:=\{z\in\mathbb{R}^N\,|\, \varphi(t,z)\to0 \text{ as } t\to -\infty\},
\end{align*}
where $\varphi(t,z)$ is the solution of (\ref{eqn:dyn_sys}) starting $z$ at $t=0$. It is known that $S$, $U$ are invariant under the flow of $f$. It holds that 
\begin{subequations}\label{ineq:flows_on_u_stablemani}
\begin{align}
    |\varphi(t,z_0)| & \leqslant K e^{-a t} \text{ for } t\geqslant0\\
    |\varphi(t,z_1)| & \leqslant K e^{a t} \text{ for } t\leqslant0, 
\end{align}
\end{subequations}
where $K>0$ is a constant dependent on $z_0$ and $z_1$ and $a>0$ is a constant independent of $z_0$ and $z_1$. See, e.g., \cite{Hale:73:ODE,Palis:82:GTDS} for more detail on the theory of stable manifold. Next Proposition, which is taken from \cite[Proposition~2.2]{sakamoto:19:cdc} and proved using the $\lambda$-lemma (see, e.g., \cite{Palis:82:GTDS}), describes more detailed behaviors of solutions near the stable and unstable manifolds. 
\begin{proposition}\label{prop:pre-turnpike} Suppose that $f$ satisfies Assumption~\ref{assm:hyp} and take $K$, $a$ in (\ref{ineq:flows_on_u_stablemani}). Then the following hold. 
\begin{enumerate}[(i)]
    \item There exists a $T_0>0$ such that for every $T>T_0$ there exists a $\rho>0$ such that 
    \[
    |\varphi(t,y)|\leqslant Ke^{-a t} \text{ for } t\in[0,T],\ y\in B(z_0,\rho), 
    \]
    where $B(x_0,\rho)$ is the $N$-dimensional ball centered at $z_0$ with radius $\rho$. Moreover, $\rho\to0$ when $T\to\infty$.
\item There exist a $T_0<0$ such that for every $T<T_0$ there exists a $\rho>0$ such that 
    \[
    |\varphi(t,y)|\leqslant Ke^{a t} \text{ for } t\in[T,0],\ y\in B(z_1,\rho), 
    \]
Moreover, $\rho\to0$ when $T\to -\infty$.
\item For any $(N-k)$-dimensional disc $\bar D$
transversal to $S$ at $z_0$ and any $k$-dimensional disc $\bar E$ transversal to $U$ at $z_1$, there exists a $T_0>0$ such that for any $T>T_0$ there exist an $(n-k)$-dimensional disc $D\subset \bar D$ transversal to $S$ at $z_0$ and a $k$-dimensional disc $E\subset \bar E$ transversal to $U$ at $z_1$ such that $\varphi(T,D)$ intersects $\varphi(-T,E)$ at a single point. 
\end{enumerate}
\end{proposition}
The above Proposition is used to prove the following result which shows that turnpike-like behaviors can be observed in a general hyperbolic dynamical systems. 
\begin{theorem}\label{thm:turnpk_in_dyn}
Suppose that $f$ satisfies Assumption~\ref{assm:hyp}. Then, for any $z_0\in S$, any $z_1\in U$, any $(N-k)$-dimensional disc $\bar D$
transversal to $S$ at $z_0$ and any $k$-dimensional disc $\bar E$ transversal to $U$ at $z_1$, there exists a $T_0>0$ such that for every $T>T_0$ there exist $\rho>0$, $y_0\in B(z_0,\rho)\cap \bar D$ and $y_1\in B(z_1,\rho)\cap \bar E$ such that $\varphi(T,y_0)=y_1$ and 
\[
|\varphi(t,y_0)|\leqslant K\left[e^{-a t}+e^{-a(T-t)}\right] \text{ for } t\in [0,T].
\]
Moreover, $\rho\to0$ when $T\to\infty$. 
\end{theorem}
We refer to \cite[Figure~1]{sakamoto:19:cdc} for the geometric interpretation and proof of Theorem~\ref{thm:turnpk_in_dyn}. 
\subsection{Turnpike analysis for maximum hands-off control}
%
\begin{theorem}\label{thm:mho_tp}
Assume that $A$ has no eigenvalues on the imaginary axis. Suppose also that Assumption~\ref{assm:normal} holds and that $x_0\in \mathscr{L}^{-}(A)$, $x_f\in \mathscr{L}^{+}(A)$. Then, for sufficiently large $T>0$, the maximum hands-off control $u_T(t)$ exists for (\ref{eqn:lti}) and satisfies 
\begin{equation}
    |u_T(t)|+|x_T(t,x_0)|\leqslant K\left[e^{-a t}+e^{-a(T-t)}\right] \ \text{for }t\in[0,T]\label{eqn:tp_ineq_mho}
\end{equation}
where $K>0$ and $a>0$ are constants independent of $T$ and $x_T(t,x_0)$
is the corresponding solution to (\ref{eqn:lti}). Moreover, when $T\to\infty$, the maximum hands-off control tends to two maximum hands-off controls, one of which takes the states from $x_0$ to the origin and the other takes them from the origin to $x_f$. 
\end{theorem}
\begin{proof}
(Step 1) It has been shown, in the proof of Theorem~\ref{thm:existence_sparse_opt}, that the maximum hands-off control exists which is also $L^1$ optimal for Problem~\ref{pbm:l1_optimal} with $Q=0$. 
From the necessary condition, there exist $x(t)$, $p(t)$ on $[0,T]$ satisfying 
\begin{subequations}\label{eqn:ham_sys}
\begin{align}
    \dot{x}&=Ax +B\mathrm{\bf dz}(B^\top p)\\
    \dot{p} &=-A^\top p
\end{align}
\end{subequations}
with $x(0)=x_0$ and $x(T)=x_f$, where 
\begin{gather*}
\mathrm{\bf dz}(x) =\bmat{\mathrm{dz}(x_1),\ldots,\mathrm{dz}(x_n)}^\top,\\
\begin{aligned}
\mathrm{dz}(w)&=\begin{cases}-1\ (\text{if }w<-1)\\ 0\ (\text{if }-1<w<1)\\ 1\ (\text{if }1<w ) \end{cases}\\
\mathrm{dz}(w) & \in [-1,0]\ (\text{if }w=-1)\\
\mathrm{dz}(w) &\in [0,1]\ (\text{if }w=1)
\end{aligned}
\end{gather*}
is the dead-zone function which is a set-valued function. The optimal control $u^\ast$ is written with $x(t)$, $p(t)$ in (\ref{eqn:ham_sys}) as
\begin{equation}
u^\ast(t)=\mathrm{\bf dz}(B^\top p(t)).\label{eqn:opt_u_dz}
\end{equation}
This step is a simple restatement of the results in \cite[Chapter~6]{Athans:66:OC}. \\
(Step 2) We show that (\ref{eqn:ham_sys}) satisfies Assumption~\ref{assm:hyp}. 
For sufficiently small $x$, $p$, (\ref{eqn:ham_sys}) is $C^1$ and from the spectral conditions on $A$, it is hyperbolic. We now consider initial value problems for (\ref{eqn:ham_sys}) and prove that it admits unique Carath\'eodory solutions. Take an arbitrary point $(\xi_0,\eta_0)\in\mathbb{R}^{2n}$ as an initial condition for (\ref{eqn:ham_sys}). The normality condition (Assumption~\ref{assm:normal}) means that the set of times on which $B^\top p(t)=\pm1$ holds has Lebesgue measure 0 (\cite[Chapter~6]{Athans:66:OC}). Therefore, one sees that  
\[
\dot x=Ax+B\mathrm{\bf dz}(B^\top \exp(-A^\top t)\eta_0)
\]
satisfies the Carath\'eodory condition for existence and uniqueness for any initial conditions. It is also seen that the existence domain is $\mathbb{R}$. \\
(Step 3) Let $S$ and $U$ be stable and unstable manifolds of (\ref{eqn:ham_sys}) at $(x,p)=(0,0)$, respectively. One sees that for initial point $(\xi_0,0)$ with $\xi_0\in\mathscr{L}^-(A)$, the corresponding solution satisfies $(x(t),p(t))\to0$ as $t\to\infty$ since $p(t)\equiv0$ and therefore, $(\xi_0,0)\in S$ for $\xi_0\in\mathscr{L}^-(A)$. Similarly, we have $(\xi_f,0)\in U$ for $\xi_f\in\mathscr{L}^+(A)$. Now, to apply Theorem~\ref{thm:turnpk_in_dyn}, let $z_0=(x_0,0)$, $z_1=(x_f,0)$ and consider 
\begin{align*}
\bar{D}\cap B(z_0,\rho)&=\{(x_0,p)\,|\,|p|<\rho\}\\
\bar{E}\cap B(z_1,\rho)&=\{(x_f,p)\,|\,|p|<\rho\}. 
\end{align*}
Let $\varphi(t,(x,p))$ be the solution of (\ref{eqn:ham_sys}) starting from $(x,p)$ at $t=0$. The theorem says that for sufficiently large $T$, there exist $p_0$, $p_f$ and $\rho>0$ with $(x_0,p_0)\in \bar{D}\cap B(z_0,\rho)$ and $(x_f,p_f)\in \bar{E}\cap B(z_0,\rho)$ such that $\varphi(T,(x_0,p_0))=(x_f,p_f)$, namely, a solution of a 2-point boundary value problem, and 
\[
|x_T(t)|+|p_T(t)|\leqslant K\left[e^{-a t}+e^{-a(T-t)}\right]\ \text{for }t\in[0,T],
\]
where we have written $(x_T(t),p_T(t))=\varphi(t,(x_0,p_0))$ and $K$, $a$ are independent of $T$. From (\ref{eqn:opt_u_dz}), we obtain (\ref{eqn:tp_ineq_mho}) by properly changing $K$ if necessary. The last statement is shown from the last one in Theorem~\ref{thm:turnpk_in_dyn} noting that $\rho\to0$ implies $y_0\to z_0$ and $y_1\to z_1$. 
\end{proof}
\begin{remark}
\begin{enumerate}
    \item The occurrence of turnpike in Theorem~\ref{thm:mho_tp} depends on the locations of $x_0$ and $x_f$ (subspaces they belong to). This is due to the constraint $|u|\leqslant1$ imposed on the maximum hands-off control problem. 
    \item The spectral condition on $A$ is necessary to apply Theorem~\ref{thm:turnpk_in_dyn} which essentially relies on the hyperbolic nature of dynamical systems. 
\end{enumerate}
\end{remark}

%% file: Simulation.tex
\section{Simulation}\label{sctn:simulation}
In this section, we show simulation to illustrate the properties of maximum hands-off control that have been proved in the previous sections.
We consider the linear system given in \eqref{eqn:lti} with
\begin{equation}
    A = \begin{bmatrix}1&1\\0&-1\end{bmatrix},\quad B = \begin{bmatrix}1\\1\end{bmatrix}.
\end{equation}
We here assume a single input (i.e., $m=1$) for simplicity.
It is easily checked that $(A,B)$ is controllable.
For this system, we have 
\begin{equation}
\begin{split}
    \mathscr{L}^{-}(A) &=\left\{\begin{bmatrix}x_1\\x_2\end{bmatrix}\in\mathbb{R}^2\,\big|\, 2x_1+x_2=0\right\},\\
    \mathscr{L}^{+}(A) &=\left\{\begin{bmatrix}x_1\\x_2\end{bmatrix}\in\mathbb{R}^2\,\big|\, x_2=0\right\}.
\end{split}
\end{equation}
We set the initial and terminal states as follows:
\begin{equation}
    x_0 = \begin{bmatrix}1\\-2\end{bmatrix}\in S,\quad
    x_f = \begin{bmatrix}1\\0\end{bmatrix}\in U.
\end{equation}

For this system, we first compute the maximum hands-off control, the solution to Problem \ref{pbm:mhandsoff},
with $T=2$.
Figure \ref{fig:control} shows the optimal control.
\begin{figure}[t]
    \centering
    \includegraphics[width=\linewidth]{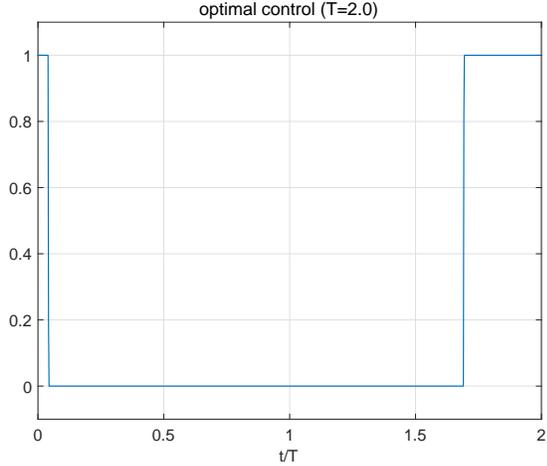}
    \caption{Maximum hands-off control $u(t)$ with $T=2$.}
    \label{fig:control}
\end{figure}
Note that this is obtained by $L^1$ optimization, which is equivalent to the
$L^0$ optimal solution since $A$ is non-singular \cite{Nagahara:16:ieeetac}.
We can see the control is sufficiently sparse, namely, $u(t)=0$ for 
$t\in(0.046,1.69)$. In fact, we have $\|u\|_0\approx0.356\ll 2$.

Next, we show the turnpike property of the maximum hands-off control for this system.
We compute the optimal controls for $T=2,4,8,16,32$
by solving the associated $L^1$ optimal control problems.
Note again that since $A$ is non-singular, the $L^1$ optimal solutions are
also $L^0$ optimal.
Figure \ref{fig:states} shows the state trajectories with the optimal controls.
\begin{figure}[t]
    \centering
    \includegraphics[width=\linewidth]{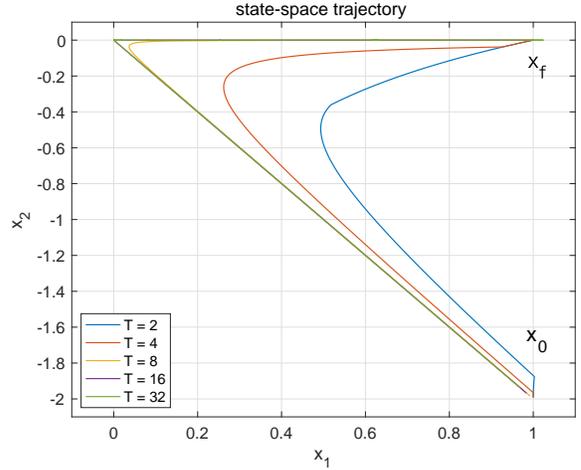}
    \caption{State-space trajectories by the maximum hands-off control with $T=2,4,8,16,32$.}
    \label{fig:states}
\end{figure}
We can see that as $T$ becomes larger, the trajectory from $x_0$ to $x_f$
approaches closer to the origin in the middle of the path.
Also, Figure \ref{fig:state_mag} shows the magnitude $\|x(t)\|_2=\sqrt{x_1(t)^2+x_2(t)^2}$ of the controls.
\begin{figure}[t]
    \includegraphics[width=\linewidth]{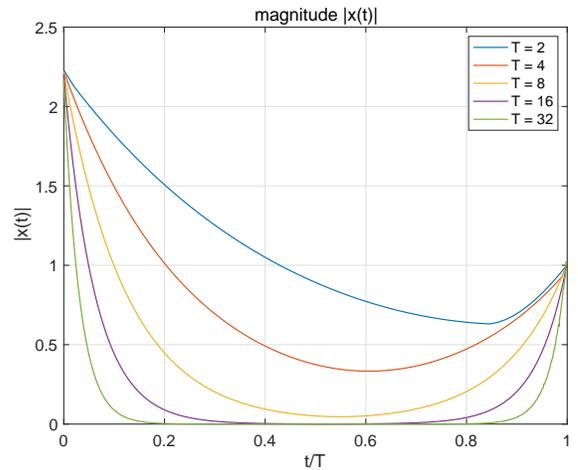}
    \caption{Magnitude $\|x(t)\|_2$ of the states with $T=2,4,8,16,32$.}
    \label{fig:state_mag}
\end{figure}
In this figure, we normalize the time axis as $t/T$ for the comparison of
time duration on which $\|x(t)\|_2\approx 0$.
For larger horizon length $T$, the control stays around the origin for a longer time duration.
These results well illustrate the turnpike property discussed in Section \ref{sctn:HHO_TP}.